\documentclass{amsart}
\usepackage{graphicx}
\usepackage[top=1.5in,bottom=1.5in,left=1.5in,right=1.5in]{geometry}
\usepackage{leftidx}
\usepackage{amssymb}
\newtheorem{theorem}{Theorem}[section]
\newtheorem*{theorem A}{Theorem A}
\newtheorem*{theorem B}{N\"olker's Theorem}
\newtheorem{lemma}{Lemma}[section]

\newtheorem{corollary}{Corollary}[section]

\theoremstyle{remark}

\theoremstyle{remark}

\theoremstyle{definition}

\numberwithin{equation}{section}
\def\({\left ( }
\def\){\right )}
\def\<{\left < }
\def\>{\right >}

\begin{document}

\title[Compact gradient Einstein-type manifolds with boundary]{Compact gradient Einstein-type manifolds with boundary and constant scalar curvature}

\author{Xiaomin Chen}
\address{College of  Science, China University of Petroleum-Beijing, Beijing, 102249, China}
\email{xmchen@cup.edu.cn}
\thanks{
The author is supported by Science Foundation of China University of Petroleum-Beijing (No.2462020XKJS02, No.2462020YXZZ004).
 }


\begin{abstract}
Inspired by the study of $V$-static manifold about classification, in this article, we apply the recent results obtained by Freitas and Gomes (Compact gradient Einstein-type manifolds with boundary, 2022) to prove the rigidity results for compact gradient Einstein-type manifolds with nonempty boundary and constant scalar curvature under some suitable pinching conditions. As a special case of gradient Einstein-type manifold, we also give a rigidity result of $(m,\rho)$-quasi-Einstein manifold with boundary.
\end{abstract}



\maketitle

\section{Introduction}

Let $(M^n,g)$, $n\geq3$, be an $n$-dimensional compact Riemannian manifold. In the setting of manifolds without boundary, we say that $(M^n,g)$ be a gradient Einstein-type manifold or, equivalently, that $(M^n,g)$ supports an Einstein-type structure if there are a smooth $u$ on $M$ and a real constant $\rho\in\mathbb{R}\setminus\{0\}$ satisfying
\begin{equation}\label{1.1*}
   \alpha Ric+\beta\nabla^2u+\mu du\otimes du=(\rho R+\lambda) g,
\end{equation}
where $\alpha,\beta,\mu$ are constants with $(\alpha,\beta,\mu)\neq0$, $\lambda$ is a smooth function and $R$ is the scalar curvature of the metric $g$.  This
structure, introduced by Catino et al.\cite{CMMR}, unifies various particular cases, such as Ricci solitons, $\rho$-Einstein solitons and Yamabe solitons.
In particular, we remark that if $(\alpha,\beta,\mu,\rho)=(1,1,-\frac{1}{m},\rho),\lambda\in\mathbb{R}, 0<m\leq\infty$, namely,
\begin{equation*}
  Ric+\nabla^2u-\frac{1}{m}du\otimes du=(\rho R+\lambda) g,
\end{equation*}
then $(M^n,g)$ is called a \emph{$(m,\rho)$-quasi-Einstein manifold} (see \cite{HW}).

 By considering $f=e^{\frac{\mu}{\beta}u}$, Eq.\eqref{1.1*} is equivalent to
\begin{equation}\label{1.2*}
  \frac{\alpha}{\beta}Ric+\frac{\beta}{\mu f}\nabla^2f=\Lambda g,
\end{equation}
where $\Lambda=\frac{1}{\beta}(\rho R+\lambda)$. Nazareno and Gomes \cite{NG} proved that a nontrivial, compact, gradient Einstein-type manifold of constant scalar curvature with both $\beta$ and $\mu$ nonzero is isometric to the standard sphere $\mathbb{S}^n(r)$. 
In the insights to Eq.\eqref{1.2*}, recently Freitas and Gomes \cite{FG} studied a family of gradient Einstein-type metrics on manifolds with nonempty boundary, namely, there exists a nonconstant smooth function $f$ on $M^n$ satisfying
\begin{equation}\label{1.3}
 \left\{
   \begin{array}{ll}
     &\nabla^2f=\frac{\mu}{\beta}f(\Lambda g-\frac{\alpha}{\beta}Ric)+\gamma g, \\
     &f>0 \quad\hbox{in}\quad {\rm int}(M^n), \\
     &f=0\quad\hbox{on}\quad\partial M,
   \end{array}
 \right.
  \end{equation}
for some smooth function $\Lambda$ and constants $\alpha,\beta,\mu,\gamma$ with $\beta\neq0$. They provided a complete classification for this family of Einstein-type manifolds that are Einstein.
\begin{theorem}{\rm(\cite[Theorem 1,Theorem 2]{FG})}\label{T1}
Let $(M^n,g)$ be a compact gradient Einstein-type manifold with
connected boundary $\partial M$. If $(M^n, g)$ is an Einstein manifold, then it is
isometric to a geodesic ball in a simply connected space form for $\gamma\neq0$ or a hemisphere of a round sphere for $\gamma=0$.
\end{theorem}
For the case of manifolds with constant scalar curvature, they just gave some boundary conditions to prove that $(M^n,g)$ is an Einstein manifold (see \cite[Theorem 3, Theorem 4]{FG}).

In this paper, we consider a gradient Einstein-type manifold $(M^n,g)$ with nonempty boundary $\partial M$ such that the metric $g$ satisfies Eq.\eqref{1.3} with constant function $\Lambda$. For the convenience of calculation, we rewrite the satisfied equation as follow:
\begin{equation}\label{1.1}
 \left\{
   \begin{array}{ll}
     &\delta fRic+\nabla^2f=h g, \\
     &f>0 \quad\hbox{in}\quad {\rm int}(M^n), \\
     &f=0\quad\hbox{on}\quad\partial M,
   \end{array}
 \right.
  \end{equation}
where $\delta=\frac{\alpha\mu}{\beta^2}$ and $h=\theta f+\gamma$ with $\theta=\frac{\mu}{\beta}\Lambda$ being constant.

Observe that Eq.\eqref{1.1} is closely related to a $V$-static metric (see \cite{MT,CEM}), thus we first use the zero radial Weyl curvature, i.e. $W(\cdot,\cdot,\cdot,\nabla f)=0$ to classify gradient Einstein-type manifold. Such a condition has been used in $V$-static manifold (\cite[Corollary 1.5, Corollary 1.6]{BR}), quasi-Einstein manifold (\cite{C1}) and gradient Ricci soliton (\cite{PW}). More precisely, we prove the following theorem.
\begin{theorem}\label{T1.2*}
Let $(M^n,g)$, $n\geq4$, be a compact gradient Einstein-type manifold satisfying \eqref{1.1} with constant scalar curvature. For $-1<\delta<-\frac{n-4}{n-2}$, suppose that $M^n$ has zero radial Weyl curvature and
\begin{equation*}
  |\mathring{Ric}|\leq\frac{\delta(n-1)-2}{(n-2)\delta-n-2}\frac{R}{\sqrt{n(n-1)}},
\end{equation*}
where $\mathring{Ric}$ is the traceless Ricci tensor. If $\gamma=0$ then $(M^n,g)$ is isometric to a hemisphere of a round sphere.
\end{theorem}
Next, we adapt two different methods to give an important integral formula.
\begin{theorem}\label{T1.3}
Let $(M^n,g)$ be a compact gradient Einstein-type manifold satisfying \eqref{1.1} with  constant scalar curvature. Then the following integral formula holds:
\begin{align*}
0= &\frac{2n\delta+n-4}{2n}\int_Mf|\mathring{Ric}|^2\Delta fdV_{g} - \frac{4\delta}{n-2}\int_Mf^2 \mathring{R}_{ij}\mathring{R}_{ik}\mathring{R}_{jk}dV_{g}+2\delta\int_M f^2W_{ijkl}\mathring{R}_{ik}\mathring{R}_{jl}dV_{g}\nonumber\\
&-\delta\int_M f^2|\nabla\mathring{Ric}|^2dV_{g}+\Big(\frac{2\delta(n-1)-2}{n(n-1)}R-\theta\Big)\delta \int_Mf^2|\mathring{Ric}|^2dV_{g}-\Big[\frac{\delta(n-1)-1}{n^2}R^2 \nonumber\\
&-\theta\frac{n-1}{n}R\Big]\int_M|\nabla f|^2dV_{g}-(1-\delta)\int_M|\mathring{Ric}(\nabla f)|^2dV_{g}+\frac{1}{2}\int_M|\mathring{Ric}|^2|\nabla f|^2dV_{g},\nonumber
\end{align*}
where $W$ is the Weyl tensor.
\end{theorem}

Now, we introduce the definition of Yamabe constant on a Riemannian manifold with nonempty boundary. Given a compact $n$-dimensional Riemannian manifold $(M^n,g)$ with boundary $\partial M$, the Yamabe invariant $Y(M,\partial M,[g])$ associated to $(M^n,g)$ is defined by
\begin{align}\label{1.4}
Y(M,\partial M,[g])=\inf_{u\in W^{1,2}(M)}\frac{\int_M\Big(\frac{4(n-1)}{n-2}|\nabla u|^2+Ru^2\Big)dV_{g}+2\int_{\partial M}Hu^2dS_g}{(\int_Mu^\frac{2n}{n-2}dV_g)^\frac{n-2}{n}},
\end{align}
where $[g]$ is the conformal class of the metric $g$ and $H$ is the mean curvature of $\partial M$. For more details, we refer the readers to \cite{E}.
Catino and Baltazar et al. used the Yamabe constant to classify gradient shrinking Ricci soliton and four-dimensional Miao-Tam critical metric, respectively (see \cite{BDR,C}). Here we suppose a similar pinching condition with \cite[Theorem 2]{BDR} to obtain the following conclusion.
\begin{theorem}\label{T1.2}
Let $(M^n,g)$ be an $n$-dimensional {\rm($4\leq n\leq6$)} compact gradient Einstein-type manifold satisfying \eqref{1.1} with constant scalar curvature.
If $-1<\delta<0$ and
\begin{align*}
  \Bigg[\sqrt{\frac{n-1}{8(n-2)}}Y(M,\partial M,[g])-&\Big(\int_M\Big(|W|^2+\frac{8}{n(n-2)}|\mathring{Ric}|^2\Big)^\frac{n}{4}dV_g\Big)^\frac{2}{n}\Bigg]\Phi(M)   \\
   &\geq\frac{\delta-1}{\delta}\sqrt{\frac{(n-1)^3}{4n(n-2)}}\int_M|\mathring{Ric}|^2|\nabla f|^2dV_g,
\end{align*}
where $\Phi(M)=\Big(\int_Mf^\frac{2n}{n-2}|\mathring{Ric}|^\frac{2n}{n-2}dV_g\Big)^\frac{n-2}{n}$, then $(M^n,g)$ is isometric to a geodesic ball in a simply connected space form if $\gamma>0$, or a hemisphere of a round sphere if $\gamma=0$.
\end{theorem}
Since a $(m,\rho)$-quasi-Einstein manifold is a special gradient Einstein-type manifold, from Theorem \ref{T1.2} we obtain a rigidity result of $(m,\rho)$-quasi-Einstein manifolds with nonempty boundary.
\begin{corollary}\label{C1}
Let $(M^n,g)$ be an $n$-dimensional {\rm($4\leq n\leq6$)} compact $(m,\rho)$-quasi-Einstein manifold with nonempty boundary and constant scalar curvature. For $1<m<\infty$, if
\begin{align*}
  \Bigg[\sqrt{\frac{n-1}{8(n-2)}}Y(M,\partial M,[g])-&\Big(\int_M\Big(|W|^2+\frac{8}{n(n-2)}|\mathring{Ric}|^2\Big)^\frac{n}{4}dV_g\Big)^\frac{2}{n}\Bigg]\Phi(M)   \\
   &\geq(m+1)\sqrt{\frac{(n-1)^3}{4n(n-2)}}\int_M|\mathring{Ric}|^2|\nabla f|^2dV_g,
\end{align*}
then $(M^n,g)$ is isometric to a hemisphere of a round sphere.
\end{corollary}
In order to prove these conclusions, in Section 2 we need review some classical tensors and give some key lemmas, and in Section 3 we will give the proof of our results.
\section{Preliminaries}
In this section we shall collect some fundamental identities and results that will be used in the proof of our results.
Recall that on an $n$-dimensional Riemannian manifold $(M^n,g)$ for $n\geqslant3$, the Weyl tensor and the Cotton tensor are respectively defined by
\begin{align}\label{2.1*}
W_{ijkl}=&R_{ijkl}-\frac{1}{n-2}(R_{ik}g_{jl}+R_{jl}g_{ik}-R_{il}g_{jk}-R_{jk}g_{il})\\
&+\frac{R}{(n-1)(n-2)}(g_{jl}g_{ik}-g_{il}g_{jk})\nonumber
\end{align}
and
\begin{equation}\label{2.2**}
C_{ijk}=\nabla_iR_{jk}-\nabla_jR_{ik}-\frac{1}{2(n-1)}(g_{jk}\nabla_iR-g_{ik}\nabla_jR).
\end{equation}

Notice that $C_{ijk}$ is skew-symmetric in the first two indexes and trac-free in any index, i.e.
\begin{equation*}
  C_{ijk}=-C_{jik}\quad\hbox{and}\quad C_{iik}=C_{iji}=0.
\end{equation*}
When $n\geq4$, the Cotton tensor and Weyl tensor satisfy the following relation:
\begin{equation}\label{2.3**}
  C_{ijk}=-\frac{n-2}{n-3}\nabla_lW_{ijkl}.
\end{equation}
For a tensor $T$, we denote by $$\mathring{T}=T-\frac{trT}{n}g$$ the traceless part of $T$.

\begin{lemma}
Let $(M^n,g)$ be a compact gradient Einstein-type manifold satisfying \eqref{1.1} with constant scalar curvature. Then we have:
\begin{align}
  \delta f(\nabla_iR_{jk}-\nabla_jR_{ik})&=-R_{ijkl}\nabla_lf-\delta(\nabla_ifR_{jk}-\nabla_jfR_{ik})+\nabla_ihg_{jk}-\nabla_jhg_{ik}\label{2.1},\\
  (1+\delta)R_{jl}\nabla_lf&=\delta\nabla_jfR+(1-n)\nabla_jh.\label{2.3*}
\end{align}
\end{lemma}
\begin{proof}
Taking the convariant derivative of \eqref{1.1}, we obtain
\begin{equation*}
  \delta (\nabla_ifR_{jk}+f\nabla_iR_{jk})+\nabla_i\nabla_j\nabla_k f=\nabla_ihg_{jk}.
\end{equation*}
Using the formula for the commutation of derivatives and Ricci identity
\begin{equation*}
  \nabla_i\nabla_j\nabla_k f-\nabla_j\nabla_i\nabla_k f=R_{ijkl}\nabla_lf,
\end{equation*}
we get the desired equation \eqref{2.1}.
Moreover, since the scalar curvature $R$ is constant, $\nabla_iR_{ij}=\frac{1}{2}\nabla_jR=0$. Thus letting $i=k$ in \eqref{2.1} and contracting it will give \eqref{2.3*}.
\end{proof}
\begin{lemma}\label{L2.2*}
Let $(M^n,g)$ be a compact gradient Einstein-type manifold satisfying \eqref{1.1} with constant scalar curvature. Then for $-1<\delta<0$ we have
 \begin{equation}\label{2.4*}
  \theta\geq\frac{(n-1)\delta-1}{n(n-1)}R.
 \end{equation}
In addition, if $\gamma\geq0$ then $\theta\leq\frac{\delta R}{n}$ and $R>0$.
\end{lemma}
\begin{proof}
Differentiating covariantly \eqref{2.3*} gives
\begin{align*}
  (1+\delta)R_{jl}\nabla_j\nabla_lf=&\delta\Delta fR+(1-n)\Delta h.
\end{align*}
Using \eqref{1.1}, we have $\Delta f=(-\delta R+n\theta)f+n\gamma$ and $\Delta h=\theta\Delta f$, thus the above formula becomes
\begin{align*}
  (1+\delta)(-\delta |Ric|^2+\theta R)f+(1+\delta)\gamma R=&(\delta R+(1-n)\theta)\Delta f\\
=&(\delta R+(1-n)\theta)[(-\delta R+n\theta)f+n\gamma],
\end{align*}
that is,
\begin{equation}\label{2.5*}
  (n-1)\Big[(n\theta-\delta R)f+n\gamma\Big]\Big[\theta-\frac{(n-1)\delta-1}{n(n-1)}R\Big]=\delta(1+\delta)|\mathring{Ric}|^2f.
\end{equation}
Here we have used $|Ric|^2=|\mathring{Ric}|^2+\frac{R^2}{n}$.

As $f$ vanishes on the boundary, we have
\begin{align*}
  0=&\int_Mdiv(f\nabla f)dV_{g}=\int_Mf\Delta fdV_{g}+\int_M|\nabla f|^2dV_{g}\\
=&\int_Mf\Big[(n\theta-\delta R)f+n\gamma\Big]dV_{g}+\int_M|\nabla f|^2dV_{g},
\end{align*}
that is,
\begin{equation}\label{2.8}
\int_M\Big[(n\theta-\delta R)f+n\gamma\Big]fdV_{g}=-\int_M|\nabla f|^2dV_{g}.
\end{equation}
For $-1<\delta<0$, integrating \eqref{2.5*} over $M$ and using \eqref{2.8}, we have
\begin{equation*}
  -(n-1)\Big[\theta-\frac{(n-1)\delta-1}{n(n-1)}R\Big]\int_M|\nabla f|^2dV_{g}=\delta(1+\delta)\int_M|\mathring{Ric}|^2f^2dV_{g}\leq0,
\end{equation*}
which yields \eqref{2.4*}. Furthermore, if $\gamma\geq0$ then from \eqref{2.8} we obtain
\begin{equation*}
  (n\theta-\delta R)\int_Mf^2dV_{g}=-n\gamma\int_MfdV_{g}-\int_M|\nabla f|^2dV_{g}\leq0,
\end{equation*}
that means that $\frac{(n-1)\delta-1}{n(n-1)}R\leq\theta\leq\frac{\delta R}{n}$. This shows $R>0$.
Therefore we complete the proof.
\end{proof}

\begin{lemma}\label{L2.2}
Let $(M^n,g)$ be a compact gradient Einstein-type manifold satisfying \eqref{1.1} with constant scalar curvature. Then we have:
\begin{align*}
\frac{1}{2}div(f\nabla|Ric|^2)=&(1-\delta)\nabla_i(fC_{ijk}R_{jk})-\frac{\delta-1}{2}\langle\nabla f,\nabla|Ric|^2\rangle\\
&-\frac{(n-1)\delta-1}{n}|Ric|^2\Delta f-(\delta+1) f\frac{(n-2)\delta-n}{n-2} R_{ij}R_{ik}R_{jk}\nonumber\\
&-\frac{1-\delta}{2}f|C|^2+f|\nabla Ric|^2-(\delta+1) f\Big[W_{ijkl}R_{ik}R_{jl}\nonumber\\
&-\frac{R^3}{(n-1)(n-2)}-\frac{\delta(n-1)(n-2)-n(2n-1)}{n(n-1)(n-2)}R|Ric|^2\Big]\nonumber\\
&+\Delta h R-\nabla_i\nabla_jhR_{ij}.\nonumber
\end{align*}
\end{lemma}
\begin{proof}
Since $R$ is constant, by \eqref{2.2**} and \eqref{2.1} we compute
\begin{align}\label{2.2}
  &\nabla_i(-\delta\nabla_jfR_{ik}R_{jk}+R_{ijkl}\nabla_lfR_{jk}) \\
=&-\delta\nabla_i(\nabla_jfR_{ik}R_{jk})+\nabla_i\Big(-\delta fC_{ijk}R_{jk}\nonumber\\
&-\delta(\nabla_if|Ric|^2-\nabla_jfR_{ik}R_{jk})+\nabla_ihR-\nabla_jhR_{ij}\Big)\nonumber\\
  = & -\delta\nabla_i(fC_{ijk}R_{jk})-\delta\langle\nabla f,\nabla|Ric|^2\rangle-\delta|Ric|^2\Delta f+\Delta h R-\nabla_i\nabla_jhR_{ij}.\nonumber
\end{align}
Here we have used $\nabla_iR_{ij}=\frac{1}{2}\nabla_jR=0$.

At the same time, by \eqref{1.1} we also have
\begin{align}\label{2.3}
  &\nabla_i(-\delta\nabla_jfR_{ik}R_{jk}+R_{ijkl}\nabla_lfR_{jk}) \\
=&-\delta(\nabla_i\nabla_jfR_{ik}R_{jk}+\nabla_jfR_{ik}\nabla_iR_{jk})+\nabla_iR_{ijkl}\nabla_lfR_{jk}\nonumber\\
&+R_{ijkl}\nabla_i\nabla_lfR_{jk}+R_{ijkl}\nabla_lf\nabla_iR_{jk}\nonumber\\
=&-\delta(-\delta fR_{ij}R_{ik}R_{jk}+h|Ric|^2+\nabla_jfR_{ik}\nabla_iR_{jk})+\nabla_iR_{ijkl}\nabla_lfR_{jk}\nonumber\\
&-\delta fR_{ijkl}R_{il}R_{jk}-h|Ric|^2+R_{ijkl}\nabla_lf\nabla_iR_{jk}.\nonumber
\end{align}
From \eqref{2.1*}, using \eqref{2.3**} one can verify
\begin{align*}
  \nabla_iR_{ijkl}= & \nabla_iW_{ijkl}+\frac{1}{n-2}(\nabla_iR_{jl}g_{ik}-\nabla_iR_{jk}g_{il}) \\
  = &-\frac{n-3}{n-2}C_{kjl}+\frac{1}{n-2}(\nabla_kR_{jl}-\nabla_lR_{jk})\\
=&-C_{kjl}
\end{align*}
and from \eqref{2.1} we obtain
\begin{align*}
  R_{ijkl}\nabla_lf\nabla_iR_{jk}= &\Big(-\delta f(\nabla_iR_{jk}-\nabla_jR_{ik}) -\delta(\nabla_ifR_{jk}-\nabla_jfR_{ik})\\
&+\nabla_ihg_{jk}-\nabla_jhg_{ik}\Big)\nabla_iR_{jk} \\
  = &-\delta fC_{ijk}\nabla_iR_{jk}-\frac{\delta}{2}\langle\nabla f,\nabla|Ric|^2\rangle+\delta\nabla_jfR_{ik}\nabla_iR_{jk}\\
=&-\frac{\delta f}{2}|C|^2-\frac{\delta}{2}\langle\nabla f,\nabla|Ric|^2\rangle+\delta\nabla_jfR_{ik}\nabla_iR_{jk}.
\end{align*}
By the skew-symmetric of $C_{ijk}$, substituting the above two equations into \eqref{2.3} implies
\begin{align}\label{2.7*}
  &\nabla_i(-\delta\nabla_jfR_{ik}R_{jk}+R_{ijkl}\nabla_lfR_{jk}) \\
= & \delta f(\delta R_{ij}R_{ik}R_{jk}+R_{ijkl}R_{ik}R_{jl})+C_{ijk}\nabla_jfR_{ik}\nonumber\\
&-\frac{\delta f}{2}|C|^2-\frac{\delta}{2}\langle\nabla f,\nabla|Ric|^2\rangle-(\delta+1)h|Ric|^2.\nonumber
\end{align}

On the other hand, making use of Ricci identity and \eqref{2.2**}, a straightforward calculation gives (see \cite[Eq.(3-3)]{BR})
\begin{align}\label{2.4}
  \frac{1}{2}f|C|^2= & f|\nabla Ric|^2+C_{ijk}\nabla_jfR_{ik}+\frac{1}{2}\langle\nabla f,\nabla|Ric|^2\rangle \\
   &+f(R_{ij}R_{jk}R_{ki}-R_{ijkl}R_{ik}R_{jl})-\nabla_j(f\nabla_iR_{jk}R_{ik}).\nonumber
\end{align}
Thus inserting \eqref{2.4} into \eqref{2.7*}, we conclude
\begin{align}\label{2.13**}
  &\nabla_i(-\delta\nabla_jfR_{ik}R_{jk}+R_{ijkl}\nabla_lfR_{jk}) \\
  = & (\delta+1) f((\delta-1) R_{ij}R_{ik}R_{jk}+R_{ijkl}R_{ik}R_{jl})+\frac{1-\delta}{2}f|C|^2-f|\nabla Ric|^2\nonumber\\
&-\frac{\delta+1}{2}\langle\nabla f,\nabla|Ric|^2\rangle-(\delta+1)h|Ric|^2+\nabla_j(f\nabla_iR_{jk}R_{ik})\nonumber\\
 = & (\delta+1) f((\delta-1) R_{ij}R_{ik}R_{jk}+R_{ijkl}R_{ik}R_{jl})+\frac{1-\delta}{2}f|C|^2-f|\nabla Ric|^2\nonumber\\
&-\frac{\delta+1}{2}\langle\nabla f,\nabla|Ric|^2\rangle-(\delta+1)h|Ric|^2+\nabla_j(fC_{ijk}R_{ik})+\frac{1}{2}div(f\nabla|Ric|^2).\nonumber
\end{align}

Now we combine \eqref{2.13**} and \eqref{2.2} to get
\begin{align}\label{2.7}
\frac{1}{2}div(f\nabla|Ric|^2)=&(1-\delta)\nabla_i(fC_{ijk}R_{jk})-\frac{\delta-1}{2}\langle\nabla f,\nabla|Ric|^2\rangle-\delta|Ric|^2\Delta f\\
 & -(\delta+1) f((\delta-1) R_{ij}R_{ik}R_{jk}+R_{ijkl}R_{ik}R_{jl})\nonumber\\
&-\frac{1-\delta}{2}f|C|^2+f|\nabla Ric|^2+(\delta+1)h|Ric|^2+\Delta h R-\nabla_i\nabla_jhR_{ij}.\nonumber
\end{align}

Contracting equation \eqref{1.1} gives $h=\frac{1}{n}(\Delta f+\delta fR)$, thus \eqref{2.7} becomes
\begin{align*}
\frac{1}{2}div(f\nabla|Ric|^2)=&(1-\delta)\nabla_i(fC_{ijk}R_{jk})-\frac{\delta-1}{2}\langle\nabla f,\nabla|Ric|^2\rangle-\frac{(n-1)\delta-1}{n}|Ric|^2\Delta f\\
 & -(\delta+1) f((\delta-1) R_{ij}R_{ik}R_{jk}+R_{ijkl}R_{ik}R_{jl})-\frac{1-\delta}{2}f|C|^2+f|\nabla Ric|^2\nonumber\\
&+\frac{(\delta+1)\delta fR}{n}|Ric|^2+\Delta h R-\nabla_i\nabla_jhR_{ij}.\nonumber
\end{align*}
Finally, taking account \eqref{2.1*} into the above equation, we get the desired equation.
\end{proof}
Next we remember that Baltazar-Ribeiro.JR obtained the following divergent formula for any Riemannian manifold with constant scalar curvature.
\begin{lemma}{\rm(\cite[Lemma 3.1]{BR})}\label{L2.3}
 Let $(M^n, g)$ be a connected Riemannian manifold with constant scalar
curvature and $f: M\rightarrow \mathbb{R}$ be a smooth function defined on $M.$ Then we have
\begin{align*}
div (f\nabla|Ric|^2)=&-f|C|^2+2f|\nabla Ric|^2+\langle\nabla f,\nabla|Ric|^2\rangle+\frac{2n}{n-2}fR_{ij}R_{ik}R_{jk}\\
&-\frac{4n-2}{(n-1)(n-2)}fR|\mathring{Ric}|^2-\frac{2}{n(n-2)}f R^3\\
&+2\nabla_i(fC_{ijk} R_{jk})+2C_{ijk}\nabla_jfR_{ik}-2f W_{ijkl}R_{ik} R_{jl}.
\end{align*}
\end{lemma}
By Lemma \ref{L2.2} and Lemma \ref{L2.3}, we have
\begin{lemma}\label{L2.4}
Let $(M^n,g)$ be a compact gradient Einstein-type manifold satisfying \eqref{1.1} with constant scalar curvature. Then we have:
\begin{align}\label{2.16*}
\frac{\delta}{2} div(f\nabla|Ric|^2)=&-\frac{\delta(n-1)-1}{n}|Ric|^2\Delta f -\frac{(n-2)\delta-n-2}{n-2}\delta f \mathring{R}_{ij}\mathring{R}_{ik}\mathring{R}_{jk}\\
&+\delta f|\nabla Ric|^2-\frac{2\delta(n-1)-2}{n(n-1)}\delta fR|\mathring{Ric}|^2+\Delta h R-\nabla_i\nabla_jhR_{ij}\nonumber\\
&-(1-\delta)C_{ijk}\nabla_jfR_{ik}-2\delta fW_{ijkl}R_{ik}R_{jl}.\nonumber
\end{align}
\end{lemma}
\begin{proof}Since  $|Ric|^2=|\mathring{Ric}|^2+\frac{R^2}{n}$, combining Lemma \ref{L2.2} and Lemma \ref{L2.3} to remove the term $\nabla_i(fC_{ijk} R_{jk})$, we conclude
\begin{align*}
\delta div(f\nabla|Ric|^2)=&-\frac{2(n-1)\delta-2}{n}|Ric|^2\Delta f -\delta f\frac{2(\delta+1)(n-2)-4n}{n-2} R_{ij}R_{ik}R_{jk}\\
&+\frac{2(\delta+1)(n-1)(n-2)-4n(2n-1)}{n(n-1)(n-2)}\delta fR|\mathring{Ric}|^2+2\delta f|\nabla Ric|^2\\
&+2\Delta h R-2\nabla_i\nabla_jhR_{ij}-2(1-\delta)C_{ijk}\nabla_jfR_{ik}\\
&-4\delta fW_{ijkl}R_{ik}R_{jl}+\frac{2(\delta+1)(n-2)-4n}{n^2(n-2)}\delta fR^3.
\end{align*}
Using $Ric=\mathring{Ric}-\frac{R}{n}g$, a direct computation yields
\begin{equation}\label{2.16}
  R_{ij}R_{ik}R_{jk}=\mathring{R}_{ij}\mathring{R}_{ik}\mathring{R}_{jk}+\frac{3}{n}R|\mathring{Ric}|^2+\frac{R^3}{n^2}.
\end{equation}
Substituting \eqref{2.16} into the previous equation and a straightforward calculation, we get the desired equation \eqref{2.16*}.
\end{proof}

\section{Proof of results}
\subsection{Proof of Theorem \ref{T1.2*}}
Since $R$ is constant, in view of \eqref{2.2**}, Equation \eqref{2.1} may be rewritten as
\begin{equation*}
\delta fC_{ijk}=-R_{ijkl}\nabla_lf-\delta(\nabla_ifR_{jk}-\nabla_jfR_{ik})+\nabla_ihg_{jk}-\nabla_jhg_{ik}.
\end{equation*}
Making use of \eqref{2.1*} and \eqref{2.3*}, we thus have
\begin{align*}
  \delta fC_{ijk}=&-W_{ijkl}\nabla_lf-\frac{1}{n-2}(R_{ik}\nabla_{j}f+R_{jl}\nabla_lfg_{ik}-R_{il}\nabla_lfg_{jk}-R_{jk}\nabla_{i}f)\\
&+\frac{R}{(n-1)(n-2)}(\nabla_{j}fg_{ik}-\nabla_{i}fg_{jk})-\delta(\nabla_ifR_{jk}-\nabla_jfR_{ik})\\
&+\nabla_ihg_{jk}-\nabla_jhg_{ik}\nonumber \\
  = &-W_{ijkl}\nabla_lf-\frac{1-(n-2)\delta}{n-2}(R_{ik}\nabla_{j}f-R_{jk}\nabla_{i}f)\nonumber\\
&+\frac{1-(n-2)\delta}{(n-1)(n-2)(1+\delta)}R(\nabla_{j}fg_{ik}-\nabla_{i}fg_{jk})\\
&+\frac{1-(n-2)\delta}{(n-2)(1+\delta)}(\nabla_jhg_{ik}-\nabla_ihg_{jk}). \nonumber
\end{align*}
Hence, by the trace-free in any index of $C_{ijk}$, we have
\begin{align*}
  \delta f|C|^2= & -W_{ijks}\nabla_sfC_{ijk}-\frac{2(1-(n-2)\delta)}{n-2}R_{ik}\nabla_{j}fC_{ijk}.
\end{align*}
If $(M^n,g)$ has zero radial Weyl curvature, namely, $W_{ijks}\nabla_sf=0$, then
\begin{equation}\label{3.2}
\delta f|C|^2=-\frac{2(1-(n-2)\delta)}{n-2}R_{ik}\nabla_{j}fC_{ijk}
\end{equation}
 and it follows from \eqref{2.3**} that
\begin{align*}
  0=&\nabla_i(W_{ijkl}\nabla_kfR_{jl})\\
=&\nabla_iW_{ijkl}\nabla_kfR_{jl}+W_{ijkl}\nabla_i\nabla_kR_{jl}+W_{ijkl}\nabla_kf\nabla_iR_{jl}\\
=&\frac{n-3}{n-2}C_{klj}\nabla_kfR_{jl}-\delta fW_{ijkl}R_{ik}R_{jl},
\end{align*}
that is,
\begin{equation}\label{3.3}
  \delta fW_{ijkl}R_{ik}R_{jl}=-\frac{n-3}{n-2}C_{ijk}\nabla_jfR_{ik}.
\end{equation}

If $\gamma=0$,  we obtain $\Delta f=(-\delta R+n\theta)f$ from \eqref{1.1}. Since $h=\theta f$ with $\theta$ being constant, it is easy to see that $\Delta h=\theta\Delta f$ and $\nabla_i\nabla_jh=\theta\nabla_i\nabla_jf=\theta(-\delta fR_{ij}+hg_{ij})$.
Therefore, by integrating \eqref{2.16*} over $M$, we apply \eqref{3.2} and \eqref{3.3} to achieve
\begin{align*}
0=&\frac{\delta(n-1)-1}{n}\int_M|\mathring{Ric}|^2\Delta fdV_{g} +\frac{(n-2)\delta-n-2}{n-2}\delta\int_M f \mathring{R}_{ij}\mathring{R}_{ik}\mathring{R}_{jk}dV_{g}\\
&-\delta \int_Mf|\nabla Ric|^2dV_{g}+\frac{2\delta(n-1)-2}{n(n-1)}\delta \int_MfR|\mathring{Ric}|^2dV_{g}\\
&-[\theta\frac{n-1}{n}R-\frac{\delta(n-1)-1}{n^2}R^2]\int_M\Delta fdV_{g}-\theta\delta\int_M f|\mathring{Ric}|^2dV_{g}\\
&-\frac{(n-2)\delta+n-4}{n-2}\int_MC_{ijk}\nabla_jfR_{ik}dV_{g}\\
=& \frac{(n-2)\delta-n-2}{n-2}\delta\int_M f \mathring{R}_{ij}\mathring{R}_{ik}\mathring{R}_{jk}dV_{g}-\delta \int_Mf|\nabla Ric|^2dV_{g}\\
&+\Bigg[\Big(-\frac{\delta(n-1)(n-3)-(n-3)}{n(n-1)}R+(n-2)\theta\Big)\delta-\theta\Bigg]\int_M f|\mathring{Ric}|^2dV_{g}\\
&+\delta\frac{(n-2)\delta+n-4}{2(1-(n-2)\delta)} \int_Mf|C|^2dV_{g}.
\end{align*}

Moreover, recall that the classical Okumura's lemma \cite[Lemma 2.1]{O} implies
\begin{equation*}
  \mathring{R}_{ij}\mathring{R}_{ik}\mathring{R}_{jk}\geq-\frac{n-2}{\sqrt{n(n-1)}}|\mathring{Ric}|^3,
\end{equation*}
 and $\theta\leq\frac{\delta R}{n}$ and $R>0$ when $\gamma=0$ (see Lemma \ref{L2.2*}), thus we have
\begin{align*}
0\geq&-\frac{(n-2)\delta-n-2}{\sqrt{n(n-1)}}\delta\int_M f |\mathring{Ric}|^3dV_{g}-\delta \int_Mf|\nabla Ric|^2dV_{g}\\
&+\frac{\delta(n-1)-2}{n(n-1)}R\delta\int_M f|\mathring{Ric}|^2dV_{g}+\delta\frac{(n-2)\delta+n-4}{2(1-(n-2)\delta)} \int_Mf|C|^2dV_{g}\\
=&\int_M\Bigg[\frac{\delta(n-1)-2}{n(n-1)}R\delta-\frac{(n-2)\delta-n-2}{\sqrt{n(n-1)}}\delta|\mathring{Ric}|\Bigg]f|\mathring{Ric}|^2dV_{g}\\
&-\delta \int_Mf|\nabla Ric|^2+\delta\frac{(n-2)\delta+n-4}{2(1-(n-2)\delta)} \int_Mf|C|^2dV_{g}.
\end{align*}
Therefore under the assumption of Theorem \ref{T1.2*}, the above inequality shows $\mathring{Ric}=0$, i.e. $M$ is Einstein. So it suffices to apply Theorem \ref{T1} to conclude that $M$ is isometric to a hemisphere of a round sphere. This complete the proof.

\subsection{Proof of Theorem \ref{T1.3}} Here we shall give two methods to prove the theorem. The first method is followed from the idea of \cite[Theorem 2]{BDR}.
\begin{proof} By Lemma \ref{L2.4}, we have
\begin{align}\label{2.13*}
  \frac{\delta}{2}div(f^2\nabla|\mathring{Ric}|^2)= &\frac{\delta}{2}fdiv(f\nabla|\mathring{Ric}|^2)+\frac{\delta}{2}\langle f\nabla|\mathring{Ric}|^2,\nabla f\rangle \\
  = &f\Big[-\frac{\delta(n-1)-1}{n}|Ric|^2\Delta f -\delta f\frac{(n-2)\delta-n-2}{n-2} \mathring{R}_{ij}\mathring{R}_{ik}\mathring{R}_{jk}\nonumber\\
&+\delta f|\nabla\mathring{Ric}|^2-\frac{2\delta(n-1)-2}{n(n-1)}\delta fR|\mathring{Ric}|^2+\Delta h R-\nabla_i\nabla_jhR_{ij}\nonumber\\
&-(1-\delta)C_{ijk}\nabla_jfR_{ik}-2\delta fW_{ijkl}\mathring{R}_{ik}\mathring{R}_{jl}\Big]+\frac{\delta}{2}\langle f\nabla|\mathring{Ric}|^2,\nabla f\rangle.\nonumber
\end{align}

Noticing that $f$ vanishes on the boundary $\partial M$ and integrating over $M$ by part, one has
\begin{equation*}
  \int_M\langle f\nabla|\mathring{Ric}|^2,\nabla f\rangle dV_{g}=-\int_Mf\Delta f|\mathring{Ric}|^2dV_{g}-\int_M|\mathring{Ric}|^2|\nabla f|^2dV_{g}.
\end{equation*}
Now, integrating \eqref{2.13*} over $M$ and using the above relation, we get
\begin{align*}
  0= &\frac{\delta(n-1)-1}{n}\int_Mf|Ric|^2\Delta fdV_{g} +\delta \frac{(n-2)\delta-n-2}{n-2}\int_Mf^2 \mathring{R}_{ij}\mathring{R}_{ik}\mathring{R}_{jk}dV_{g}\\
&-\delta\int_M f^2|\nabla\mathring{Ric}|^2dV_{g}+\frac{2\delta(n-1)-2}{n(n-1)}\delta \int_Mf^2R|\mathring{Ric}|^2dV_{g}\\
&-\int_Mf\Delta h RdV_{g}+\int_Mf\nabla_i\nabla_jhR_{ij}dV_{g}\\
&+(1-\delta)\int_MfC_{ijk}\nabla_jfR_{ik}dV_{g}+2\delta\int_M f^2W_{ijkl}\mathring{R}_{ik}\mathring{R}_{jl}dV_{g}\\
&+\frac{\delta}{2}\Big(\int_Mf\Delta f|\mathring{Ric}|^2dV_{g}+\int_M|\mathring{Ric}|^2|\nabla f|^2dV_{g}\Big).
\end{align*}

For $h=\theta f+\gamma$, as before we also have that $\Delta h=\theta\Delta f$ and $\nabla_i\nabla_jh=\theta(-\delta fR_{ij}+hg_{ij})$. Substituting this into the previous equation yields
\begin{align}\label{2.13}
  0= &\frac{\delta(3n-2)-2}{2n}\int_Mf|\mathring{Ric}|^2\Delta fdV_{g}+\Big(\frac{\delta(n-1)-1}{n^2}R^2-\theta\frac{n-1}{n}R\Big)\int_Mf\Delta fdV_{g} \\
&+\delta \frac{(n-2)\delta-n-2}{n-2}\int_Mf^2 \mathring{R}_{ij}\mathring{R}_{ik}\mathring{R}_{jk}dV_{g}\nonumber\\
&-\delta\int_M f^2|\nabla\mathring{Ric}|^2dV_{g}+\frac{2\delta(n-1)-2}{n(n-1)}\delta \int_Mf^2R|\mathring{Ric}|^2dV_{g}\nonumber\\
&-\delta\theta\int_M f^2|\mathring{Ric}|^2dV_{g}+(1-\delta)\int_MfC_{ijk}\nabla_jfR_{ik}dV_{g}\nonumber\\
&+2\delta\int_M f^2W_{ijkl}\mathring{R}_{ik}\mathring{R}_{jl}dV_{g}+\frac{\delta}{2}\int_M|\mathring{Ric}|^2|\nabla f|^2dV_{g}.\nonumber
\end{align}
Recalling \eqref{2.2**} and the constancy of $R$, we compute
\begin{align*}
  C_{ijk}\nabla_jfR_{ik}=& (\nabla_iR_{jk}-\nabla_jR_{ik})\nabla_jfR_{ik}\\
=&R_{ik}\nabla_{j}f\nabla_iR_{jk}-\frac{1}{2}\langle\nabla f,\nabla|Ric|^2\rangle\\
=&\mathring{R}_{ik}\nabla_{j}f\nabla_i\mathring{R}_{jk}-\frac{1}{2}\langle\nabla f,\nabla|\mathring{Ric}|^2\rangle.
\end{align*}
Integrating this over $M$ by part, we thus have
\begin{align*}
  \int_Mf C_{ijk}\nabla_jfR_{ik}dV_{g}=& \int_Mf\mathring{R}_{ik}\nabla_{j}f\nabla_i\mathring{R}_{jk}dV_{g}-\int_M\frac{1}{2}f\langle\nabla f,\nabla|\mathring{Ric}|^2\rangle dV_{g}\\
  = &-\int_M|\mathring{Ric}(\nabla f)|^2dV_{g}-\int_Mf\mathring{R}_{ik}\mathring{R}_{jk}\nabla_i\nabla_jfdV_{g}\\
&-\int_M\frac{1}{2}f\langle\nabla f,\nabla|\mathring{Ric}|^2\rangle dV_{g}\\
   = &-\int_M|\mathring{Ric}(\nabla f)|^2dV_{g}+\int_M\delta f
^2\mathring{R}_{ik}\mathring{R}_{jk}\mathring{R}_{ij}dV_{g}\\
&+\frac{n-2}{2n}\int_Mf\Delta f|\mathring{Ric}|^2dV_{g}+\int_M\frac{1}{2}|\nabla f|^2|\mathring{Ric}|^2dV_{g}.
\end{align*}
Inserting this into \eqref{2.13} and taking account \eqref{2.8}, we thus achieve
\begin{align*}\label{}
  0= &\frac{2n\delta+n-4}{2n}\int_Mf|\mathring{Ric}|^2\Delta fdV_{g}+\Big[\frac{\delta(n-1)-1}{n^2}R^2-\theta\frac{n-1}{n}R\Big]\int_Mf\Delta f\\
 &- \frac{4\delta}{n-2}\int_Mf^2 \mathring{R}_{ij}\mathring{R}_{ik}\mathring{R}_{jk}dV_{g}-\delta\int_M f^2|\nabla\mathring{Ric}|^2dV_{g}\\
&+\frac{2\delta(n-1)-2}{n(n-1)}\delta \int_Mf^2R|\mathring{Ric}|^2dV_{g}-\delta\theta\int_M f^2|\mathring{Ric}|^2dV_{g}\\
&-(1-\delta)\int_M|\mathring{Ric}(\nabla f)|^2dV_{g}+\frac{1}{2}\int_M|\nabla f|^2|\mathring{Ric}|^2dV_{g}+2\delta\int_M f^2W_{ijkl}\mathring{R}_{ik}\mathring{R}_{jl}dV_{g}\nonumber\\
= &\frac{2n\delta+n-4}{2n}\int_Mf|\mathring{Ric}|^2\Delta fdV_{g} - \frac{4\delta}{n-2}\int_Mf^2 \mathring{R}_{ij}\mathring{R}_{ik}\mathring{R}_{jk}dV_{g}\\
&+2\delta\int_M f^2W_{ijkl}\mathring{R}_{ik}\mathring{R}_{jl}dV_{g}-\delta\int_M f^2|\nabla\mathring{Ric}|^2dV_{g}\nonumber\\
&+\Big(\frac{2\delta(n-1)-2}{n(n-1)}R-\theta\Big)\delta \int_Mf^2|\mathring{Ric}|^2dV_{g}\nonumber\\
&-\Big[\frac{\delta(n-1)-1}{n^2}R^2 -\theta\frac{n-1}{n}R\Big]\int_M|\nabla f|^2dV_{g}\\
&-(1-\delta)\int_M|\mathring{Ric}(\nabla f)|^2dV_{g}+\frac{1}{2}\int_M|\mathring{Ric}|^2|\nabla f|^2dV_{g}.\nonumber
\end{align*}
Thus the proof is complete.
\end{proof}

\begin{proof}[Another proof of Theorem \ref{T1.3}]
First we take the covariant derivative of \eqref{2.1} to achieve
\begin{align*}
  \delta f(\nabla_t\nabla_iR_{jk}-\nabla_t\nabla_jR_{ik})&=-\delta\nabla_tf(\nabla_iR_{jk}-\nabla_jR_{ik})-\nabla_tR_{ijkl}\nabla_lf-R_{ijkl}\nabla_t\nabla_lf\\
&-\delta(\nabla_t\nabla_ifR_{jk}+\nabla_if\nabla_tR_{jk}-\nabla_t\nabla_jfR_{ik}-\nabla_jf\nabla_tR_{ik})\\
&+\nabla_t\nabla_ihg_{jk}-\nabla_t\nabla_jhg_{ik}.
\end{align*}
Then letting the index $t=i$ and contracting the equation gives
\begin{align}\label{3.5}
  \delta f(\Delta R_{jk}-\nabla_i\nabla_jR_{ik})&=\delta\nabla_if\nabla_jR_{ik}-\nabla_iR_{ijkl}\nabla_lf-R_{ijkl}\nabla_i\nabla_lf\\
&-\delta(\Delta fR_{jk}+2\nabla_if\nabla_iR_{jk}-\nabla_i\nabla_jfR_{ik})+\Delta hg_{jk}-\nabla_k\nabla_jh. \nonumber
\end{align}

From the commutation relations for the second covariant derivative of $R_{ik}$, we have
\begin{equation}\label{3.6}
  \nabla_i\nabla_jR_{ik}=\nabla_j\nabla_iR_{ik}+R_{sk}R_{ijis}+R_{is}R_{ijks}=R_{sk}R_{js}+R_{is}R_{ijks}.
\end{equation}
On the other hand, from the second Bianchi identities we have
\begin{equation}\label{3.7}
  \nabla_jR_{ijkl}\nabla_lf=\nabla_lR_{ik}\nabla_lf-\nabla_kR_{il}\nabla_lf.
\end{equation}
Thus substituting \eqref{3.6} and \eqref{3.7} into \eqref{3.5}, we conclude
\begin{align*}
  \delta f\Delta R_{jk}&=\delta(\nabla_if\nabla_jR_{ik}+fR_{sk}R_{js}+2fR_{is}R_{ijks})\\
&+(\nabla_lR_{jk}-\nabla_kR_{jl})\nabla_lf+(1+\delta)hR_{jk}\\
&-\delta(\Delta fR_{jk}+2\nabla_if\nabla_iR_{jk}+\delta fR_{ij}R_{ik})+\Delta hg_{jk}-\nabla_k\nabla_jh  \\
&=\delta(\nabla_if\nabla_jR_{ik}+2fR_{is}R_{ijks})+((1-2\delta)\nabla_lR_{jk}-\nabla_kR_{jl})\nabla_lf\\
&+\Delta hg_{jk}-\nabla_k\nabla_jh+(\delta-\delta^2)fR_{ij}R_{ik}+[(1+\delta)h-\delta\Delta f]R_{jk}.
\end{align*}
Since $R$ is constant, using the above equation, we compute
\begin{align*}
  &\frac{1}{2}\delta f\Delta|\mathring{Ric}|^2=\frac{1}{2}\delta f\Delta|Ric|^2=\delta fR_{jk}\Delta R_{jk}+\delta f|\nabla\mathring{Ric}|^2\\
=&\delta(\nabla_if\nabla_jR_{ik}R_{jk}+2fR_{jk}R_{is}R_{ijks})+((1-2\delta)\nabla_lR_{jk}R_{jk}-\nabla_kR_{jl}R_{jk})\nabla_lf\\
&+\Delta hR-\nabla_k\nabla_jhR_{jk}+(\delta-\delta^2)fR_{ij}R_{ik}R_{jk}+[(1+\delta)h-\delta\Delta f]|Ric|^2+\delta f|\nabla\mathring{Ric}|^2  \\
=&2\delta fR_{jk}R_{is}R_{ijks}+\frac{1-2\delta}{2}\langle\nabla f,\nabla|Ric|^2\rangle-(1-\delta)\nabla_kR_{jl}R_{jk}\nabla_lf\\
&+\Delta hR-\nabla_k\nabla_jhR_{jk}+(\delta-\delta^2)fR_{ij}R_{ik}R_{jk}+[(1+\delta)h-\delta\Delta f]|Ric|^2+\delta f|\nabla\mathring{Ric}|^2.
\end{align*}
Moreover, recalling \eqref{2.1*} we obtain
\begin{align}\label{3.7*}
\frac{1}{2}\delta f\Delta|\mathring{Ric}|^2=&2\delta f\Big[W_{ijkl}R_{jk}R_{il}-\frac{2n-1}{(n-1)(n-2)}R|Ric|^2+\frac{R^3}{(n-1)(n-2)}\Big]\\
&+\frac{1-2\delta}{2}\langle\nabla f,\nabla|Ric|^2\rangle-(1-\delta)\nabla_kR_{jl}R_{jk}\nabla_lf\nonumber\\
&+\Delta hR-\nabla_k\nabla_jhR_{jk}+\Big(-\delta^2+\frac{(n+2)\delta}{n-2}\Big)fR_{ij}R_{ik}R_{jk}\nonumber\\
&+[(1+\delta)h-\delta\Delta f]|Ric|^2+\delta f|\nabla\mathring{Ric}|^2. \nonumber
\end{align}
Now, as we known, from Eq.\eqref{1.1} one has
\begin{align}\label{3.8}
  h=\frac{\delta Rf+\Delta f}{n} \quad\hbox{and}\quad \nabla_i\nabla_jh=\theta(-\delta fR_{ij}+hg_{ij}).
\end{align}
Hence, by \eqref{2.16} and \eqref{3.8}, Eq.\eqref{3.7*} becomes
\begin{align}\label{3.9}
  \frac{1}{2}\delta f\Delta|\mathring{Ric}|^2=&2\delta fW_{ijkl}\mathring{R}_{jk}\mathring{R}_{il}-\frac{2(n-1)\delta^2-2\delta }{n(n-1)}fR|\mathring{Ric}|^2\\
&+\frac{(1-2\delta)}{2}\langle\nabla f,\nabla|\mathring{Ric}|^2\rangle-(1-\delta)\nabla_k\mathring{R}_{jl}\mathring{R}_{jk}\nabla_lf\nonumber\\
&+\theta\frac{n-1}{n}\Delta fR+\theta\delta f|\mathring{Ric}|^2+\Big(-\delta^2+\frac{(n+2)\delta}{n-2}\Big)f\mathring{R}_{ij}\mathring{R}_{ik}\mathring{R}_{jk}\nonumber\\
&-\frac{(n-1)\delta-1}{n}\Delta f|Ric|^2+\delta f|\nabla\mathring{Ric}|^2. \nonumber
\end{align}

As $f$ vanishes on the boundary $\partial M$, we have
\begin{align}
  \int_Mf\langle\nabla f,\nabla|\mathring{Ric}|^2\rangle dV_{g}=&\frac{1}{2}\int_Mdiv(f^2\nabla|\mathring{Ric}|^2)dV_{g}-\frac{1}{2}\int_Mf^2\Delta|\mathring{Ric}|^2dV_{g}\label{3.10}\\
=&-\frac{1}{2}\int_Mf^2\Delta|\mathring{Ric}|^2dV_{g},\nonumber\\
\int_Mf\nabla_k\mathring{R}_{jl}\mathring{R}_{jk}\nabla_lfdV_{g}=&\int_M\nabla_k(f\mathring{R}_{jl}\mathring{R}_{jk}\nabla_lf)dV_{g}-\int_M|\mathring{Ric}(\nabla f)|^2dV_{g}\label{3.11}\\
&-\int_Mf\mathring{R}_{jl}\mathring{R}_{jk}\nabla_k\nabla_lfdV_{g}\nonumber\\
=&-\int_M|\mathring{Ric}(\nabla f)|^2dV_{g}-\int_Mf\mathring{R}_{jl}\mathring{R}_{jk}\nabla_k\nabla_lfdV_{g}.\nonumber
\end{align}
Multiplying \eqref{3.9} by $f$ and applying \eqref{3.10} and \eqref{3.11}, we deduce
\begin{align*}
  -\frac{1}{2}\int_Mf\langle\nabla f,\nabla|\mathring{Ric}|^2\rangle dV_{g}=&2\delta\int_M f^2W_{ijkl}\mathring{R}_{jk}\mathring{R}_{il}dV_{g}-\frac{2(n-1)\delta^2-2\delta }{n(n-1)}\int_Mf^2R|\mathring{Ric}|^2dV_{g}\\
&+(1-\delta)\Big(\int_M|\mathring{Ric}(\nabla f)|^2+\int_Mf\mathring{R}_{ik}\mathring{R}_{jk}(-\delta fR_{ij}+hg_{ij})\Big)dV_{g}\\
&+\theta\frac{n-1}{n}\int_Mf\Delta fRdV_{g}+\theta\delta\int_M f^2|\mathring{Ric}|^2dV_{g}\\
&+(-\delta^2+\frac{(n+2)\delta}{n-2})\int_Mf^2\mathring{R}_{ij}\mathring{R}_{ik}\mathring{R}_{jk}dV_{g}\\
&-\frac{(n-1)\delta-1}{n}\int_Mf\Delta f|Ric|^2dV_{g}+\delta \int_Mf^2|\nabla\mathring{Ric}|^2dV_{g},  \\
\end{align*}
that is,
\begin{align*}
0=&-2\delta\int_M f^2W_{ijkl}\mathring{R}_{ik}\mathring{R}_{jl}dV_{g}+(1-\delta)\int_M|\mathring{Ric}(\nabla f)|^2dV_{g}\\
&+\Big[\theta\delta-\frac{2(n-1)\delta^2-2\delta }{n(n-1)}R\Big]\int_M f^2|\mathring{Ric}|^2dV_{g}-\frac{1}{2}\int_M|\mathring{Ric}|^2|\nabla f|^2dV_{g}\\
&+\frac{4\delta}{n-2}\int_Mf^2\mathring{R}_{ij}\mathring{R}_{ik}\mathring{R}_{jk}dV_{g}+\frac{4-n-2n\delta}{2n}\int_Mf\Delta f|\mathring{Ric}|^2dV_{g}\\
&+\Big[\theta\frac{n-1}{n}R-\frac{(n-1)\delta-1}{n^2}R^2\Big]\int_Mf\Delta fdV_{g}+\delta \int_Mf^2|\nabla\mathring{Ric}|^2dV_{g}.
\end{align*}
Finally, by integrating by part, we also give the desired integral formula.
\end{proof}

\subsection{Proof of Theorem \ref{T1.2}}
First from \eqref{1.4} we have
    \begin{align}\label{1.2}
      \frac{n-2}{4(n-1)}Y(M,\partial M,[g])\Big(\int_M|u|^{\frac{2n}{n-2}}dV_g\Big)^\frac{n-2}{2}\leq&\int_M|\nabla u|^2dV_g+\frac{n-2}{4(n-1)}\int_MRu^{2}dV_g\\
&+\frac{n-2}{2(n-1)}\int_{\partial M}Hu^2dS_g.\nonumber
    \end{align}
for any $u\in W^{1,2}(M)$.
Using Kato inequality $|\nabla|\mathring{Ric}||^2\leq|\nabla\mathring{Ric}|^2$ and choosing $u=f|\mathring{Ric}|$ in \eqref{1.2}, Baltazar et al. proved the following inequality (see \cite[Eq.(3.16)]{BDR}):
\begin{align}\label{2.19}
  \int_M f^2|\nabla\mathring{Ric}|^2dV_{g}\geq& \frac{n-2}{4(n-1)}Y(M,\partial M,[g])\Big(\int_Mf^\frac{2n}{n-2}|\mathring{Ric}|^\frac{2n}{n-2}dV_{g}\Big)^\frac{n-2}{n}\\
&-\frac{(n-2)R}{4(n-1)}\int_Mf^2|\mathring{Ric}|^2dV_{g}+\int_Mf\Delta f|\mathring{Ric}|^2dV_{g}.\nonumber
\end{align}
Meanwhile, a straightforward computation gives (see \cite[Eq.(3.19)]{BDR})
\begin{equation}\label{2.20}
  |\mathring{Ric}(\nabla f)|^2\leq\frac{(n-1)\sqrt{2n}}{2n}|\mathring{Ric}|^2|\nabla f|^2.
\end{equation}

We remark that on every $n$-dimensional Riemannian manifold the following estimate holds (see \cite[Proposition 2.1]{C}):
\begin{equation*}
  \Big|-W_{ijkl}\mathring{R}_{ik}\mathring{R}_{jl}+\frac{2}{n-2}\mathring{R}_{ij}\mathring{R}_{jk}\mathring{R}_{ki}\Big|
\leq\sqrt{\frac{n-2}{2(n-1)}}\Big(|W|^2+\frac{8}{n(n-2)}|\mathring{Ric}|^2\Big)^\frac{1}{2}|\mathring{Ric}|^2.
\end{equation*}
As $\delta<0$, we have
\begin{equation}\label{2.21}
  -\delta W_{ijkl}\mathring{R}_{ik}\mathring{R}_{jl}+\frac{2\delta}{n-2}\mathring{R}_{ij}\mathring{R}_{jk}\mathring{R}_{ki}
\geq\delta\sqrt{\frac{n-2}{2(n-1)}}\Big(|W|^2+\frac{8}{n(n-2)}|\mathring{Ric}|^2\Big)^\frac{1}{2}|\mathring{Ric}|^2.
\end{equation}

Since $\Delta f=(-\delta R+n\theta)f+n\gamma$ from \eqref{1.1},
taking account \eqref{2.19}, \eqref{2.20} and \eqref{2.21} into Theorem \ref{T1.3} and using H\"{o}lder inequality, we follow
\begin{align}\label{2.22}
0\geq &\Bigg[\delta \sqrt{\frac{2(n-2)}{n-1}}\Big(\int_M\Big(|W|^2+\frac{8}{n(n-2)}|\mathring{Ric}|^2\Big)^\frac{n}{4}dV_{g}\Big)^\frac{2}{n}\\
&-\delta\frac{n-2}{4(n-1)}Y(M,\partial M,[g])\Bigg]\Big(\int_Mf^\frac{2n}{n-2}|\mathring{Ric}|^\frac{2n}{n-2}dV_{g}\Big)^\frac{n-2}{n}\nonumber\\
&+\Big[\Big(\frac{8\delta(n-1)-(n-4)^2}{4n(n-1)}R-\theta\Big)\delta+\frac{n-4}{2}\theta\Big] \int_Mf^2|\mathring{Ric}|^2dV_{g}\nonumber\\
&+\frac{n-4}{2}\gamma\int_Mf|\mathring{Ric}|^2dV_{g}+\Big[\theta\frac{n-1}{n}R-\frac{\delta(n-1)-1}{n^2}R^2\Big]\int_M|\nabla f|^2dV_{g} \nonumber\\
&+\Big[\theta\frac{n-1}{n}R-\frac{\delta(n-1)-1}{n^2}R^2\Big]n\gamma\int_MfdV_{g}\nonumber\\
&+\Big(\frac{1}{2}-(1-\delta)\frac{(n-1)\sqrt{2n}}{2n}\Big)\int_M|\mathring{Ric}|^2|\nabla f|^2dV_{g}.\nonumber
\end{align}
Now,  since $-1<\delta<0$, by \eqref{2.4*} we have
\begin{align*}
  &\Big(\frac{8\delta(n-1)-(n-4)^2}{4n(n-1)}R-\theta\Big)\delta+\frac{n-4}{2}\theta  \\
  \geq & \Big(\frac{8\delta(n-1)-(n-4)^2}{4n(n-1)}-\frac{(n-1)\delta-1}{n(n-1)}\Big)R\delta+\frac{(n-4)((n-1)\delta-1)}{2n(n-1)}R\\
=&\frac{4(n-1)\delta^2+(n^2-2n-4)\delta+2(n-4)}{4n(n-1)}R>0\quad\hbox{when}\;4\leq n\leq6
\end{align*}
and
\begin{align*}
  &\theta\frac{n-1}{n}R-\frac{\delta(n-1)-1}{n^2}R^2\geq0.
\end{align*}
 Hence, under the assumptions of Theorem \ref{T1.2}, from the inequality \eqref{2.22} we have $\mathring{Ric}\equiv0$, i.e. $(M^n,g)$ is an Einstein manifold.
Finally, we obtain the desired conclusion in view of Theorem \ref{T1}.
\subsection{Proof of Corollary \ref{C1}} For a $(m,\rho)$-quasi-Einstein manifold, it corresponds to the case where $\delta=-\frac{1}{m}$, $\theta=-\frac{\rho R+\lambda}{m}$ and $\gamma=0$, thus when $m>1$ and $R$ is constant,  we have $-1<\delta<0$ and $\theta$ is constant. Hence $(M^n,g)$ is isometric to a hemisphere of a round sphere by Theorem \ref{T1.2}.


\section*{Acknowledgement}
The author thanks to China Scholarship Council for supporting him to visit University of Turin as a scholar and expresses his gratitude to Professor Luigi Vezzoni and Department of Mathematics for their hospitality.

%
%

%
%
%

---------------------------------------------------------------
\end{document}